\documentclass[11pt]{amsart}
\usepackage{amsmath}
\usepackage{amssymb}
\usepackage{amsthm}
\usepackage{amscd}
\usepackage{graphicx}
\usepackage[all]{xy}

\def\acts{\curvearrowright}
\numberwithin{equation}{section}

\newcounter{AbcT}

\newtheorem {Theoremintro}    {Theorem}

\newtheorem {Theorem}    {Theorem}[section]

\def\acts{\curvearrowright}

\newtheorem {Lemma}      [Theorem]    {Lemma}
\newtheorem {Corollary}   [Theorem] {Corollary}
\newtheorem {Proposition}[Theorem]    {Proposition}
\newtheorem {Claim}      [Theorem]    {Claim}

\theoremstyle{remark}
\newtheorem {Remark}		 [Theorem]    {\bf{Remark}}

\newtheorem {Definition} [Theorem]    {\bf{Definition}}

\newcommand{\Cay}{\mathit{Cay}}

\newcounter{DM@bibnum}

\newcommand{\la}{\langle}
\newcommand{\ra}{\rangle}

\def\Stab{{\rm Stab\,}}

\def\NSL_2{{\mathcal N SL_2}}

\def\phi{\varphi}

\def\calT{{\mathcal T}}

\def\hbar{\bar h}

\def\dbN{{\mathbb N}}

\newcommand{\sub}{\subseteq}

\begin{document}

\title{Tarski numbers of group actions}
\author{Gili Golan}
\address{Department of Mathematics, Bar-Ilan University, Ramat Gan 5290002, Israel}
\email{gili.golan@math.biu.ac.il}


\begin{abstract}
The Tarski number of a group action 
$G\acts X$ is the minimal number of pieces in a paradoxical decomposition of it. 
For any $k\ge 4$ we construct a faithful transitive action of a free group of rank $k-1$ with Tarski number $k$. Using similar techniques we construct a group action $G\acts X$ with Tarski number $6$ such that the Tarski numbers of restrictions of this action to finite index subgroups of $G$ are arbitrarily large. 
\end{abstract}

\maketitle

\section{Introduction}

Let $G\acts X$ be a group action. We will always assume that groups are acting from the right. 

\begin{Definition}
The group action $G\acts X$ admits a \emph{paradoxical decomposition} if there exist positive integers $m$ and $n$, disjoint subsets $P_1,\ldots, P_m,Q_1,\ldots, Q_n$ of $X$
and subsets $S_1=\{g_1,\ldots, g_m\}$, $S_2=\{h_1,\ldots, h_n\}$ of $G$ such that
\begin{equation}\label{eq1}X=\bigcup_{i=1}^m  P_ig_i=\bigcup_{j=1}^n  Q_jh_j.\end{equation}
The sets $S_1,S_2$ are called the \emph{translating sets} of the paradoxical decomposition.
\end{Definition}

The minimal possible value of $m+n$ in a paradoxical decomposition of $G\acts X$
is called the \emph{Tarski number} of the action and denoted by $\mathcal T(G\acts X)$.
If $G$ acts on itself by right multiplication, the Tarski number of the action is called the Tarski number of $G$ and denoted by $\mathcal T(G)$.

Clearly, $m,n\ge 2$ in any paradoxical decomposition. Thus, the Tarski number of any group action cannot be smaller than $4$. By a result of Dekker (see, for example, \cite[Theorem~5.8.38]{Sa})
 the Tarski number of a group is $4$ if and only if it contains non abelian free subgroups.  Recent results about Tarski numbers of groups, show that there are groups with arbitrarily large Tarski numbers \cite{OS,EGS}. In \cite{EGS} groups with Tarski number $5$ and groups with Tarski number $6$ are constructed. Note, that no integer $\ge 7$ is known to be the Tarski number of a group.

For group actions, the situation is much more clear.

\begin{Theoremintro}\label{thm:1}
Every integer $k\ge 4$ is the Tarski number of a faithful transitive action of a finitely generated free group.
\end{Theoremintro}

To our knowledge, prior to this paper no integer $>4$ was known to be the Tarski number of a faithful action of a free group. For actions of non-free groups, the only numbers known to be Tarski numbers are $4,5$ and $6$ \cite{EGS}.
In connection with Theorem \ref{thm:1}, we mention the result of J\'onsson, characterizing group actions with Tarski number $4$.

\begin{Theorem}\label{Jonsson}\cite[Theorem~4.8]{Wag}
Let $G\acts X$ be a group action. Then $\mathcal T(G\acts X)=4$ if and only if $G$ has a non abelian free subgroup $F$ such that the restriction of the action to $F$ has cyclic point stabilizers.
\end{Theorem}

In particular, if $F$ is a non abelian free group and the action $F\acts X$ has cyclic point stabilizers then $\mathcal T(F\acts X)=\mathcal T(F)$. Part (2) of the following theorem generalizes this observation. The theorem, is the group action analogue of parts (a) and (c) of \cite[Theorem 1]{EGS}. Parts (b) and (d) can be extended to group actions as well. 

\begin{Theorem}\label{EGS}
Let $G\acts X$ be a group action.
\begin{enumerate}
\item Let $H\le G$ be a finite index subgroup and $H\acts X$ the action of $G$ restricted to $H$. Then,
$$\calT(H\acts X)-2\leq [G:H](\calT(G\acts X)-2).$$
\item If $G\acts X$ has amenable point stabilizers then $\mathcal T(G\acts X)=\mathcal T(G)$.
\end{enumerate}
\end{Theorem}

\begin{proof}
In part (2), using corollary \ref{Cor} below, one can reduce the problem to actions of $G$ on $G/\Stab_G(x)$. Then, following the argument of \cite[Theorem 1(c)]{EGS} yields the result. The proof of part (1) requires a modification of \cite[Lemma 3.1(i)]{EGS}. The definition of colored Cayley graphs from \cite{EGS} extends naturally to the group action case. Using Corollary \ref{Cor} and Lemma \ref{equi}(2) below, one can reduce the problem to finding a spanning evenly colored $k$-subgraph of $\Cay(H\acts H/\Stab_H(x),(S_1',S_2'))$, when a spanning evenly colored $k$-subgraph of $\Cay(G\acts G/\Stab_H(x),(S_1,S_2))$ is known to exist.
\end{proof}

In \cite{EGS} it is observed that there exists $t$ such that the property of having Tarski number $t$ is not invariant under quasi isometry. Indeed, a construction from \cite{EJ} yields a non amenable group $G$ with finite index subgroups with arbitrarily large Tarski numbers. 
The only estimate of the value of $t$ bounds it from above by $10^{10^8}$. We prove an analogue result for group actions with $t=6$. 

\begin{Theoremintro}\label{thm:3}
Let $F$ be a free group of rank $3$. There exists a faithful transitive action $F\acts X$ such that $\mathcal T(F\acts X)=6$ and restrictions of the action to finite index subgroups of $F$ have arbitrarily large Tarski numbers. 
\end{Theoremintro}

Note that by Theorem \ref{Jonsson}, $6$ cannot be replaced by $4$ in Theorem \ref{thm:3}. We don't know if it can be replaced by $5$. 

\vskip .12cm
{\bf Organization.}  Section \ref{background} contains background information about Tarski numbers of group actions. 
 Section \ref{automata} contains preliminary information about subgroups of free groups and their Stallings cores. Section \ref{construction} contains the proof of Theorem \ref{thm:1} and Section \ref{qi} contains the proof of Theorem \ref{thm:3}.

\vskip .12cm
{\bf Acknowledgments.} The author would like to thank Mikhail Ershov 
and Mark Sapir for useful discussions 
and comments on the text. Most of the research was done during the author's stay at Vanderbilt University and at the University of Virginia. She is grateful for the accommodations and hospitality.

\section{Tarski numbers of group actions}\label{background}

\begin{Lemma}\cite[Proposition~1.10]{Wag}
Let $G\acts X$ be a free action. Then, if $G$ has a paradoxical decomposition with translating sets $S_1,S_2$, then $X$ has a paradoxical decomposition with the same translating sets. 
\end{Lemma}

\begin{Corollary}\label{x,y}
If the free group $\la x,y\ra$  acts freely on $X$, then $X$ has a paradoxical decomposition with translating sets $\{1,x\}$ and $\{1,y\}$.
\end{Corollary}

\begin{proof}
The free group $\la x,y\ra $ has a paradoxical decomposition with these translating sets \cite[Theorem~1.2]{Wag}.
\end{proof}

\begin{Lemma}\label{equi}
Let $G\acts X$ be a group action. 
\begin{enumerate}
\item If $H\le G$ is a subgroup of $G$ and $H\acts X$ is the action of $G$ restricted to $H$ then $\mathcal T(G\acts X)\le \mathcal T(H\acts X)$.
\item Let $G\acts Y$ be another $G$-action and $f\colon X\to Y$ be a $G$-equivariant surjective map. If $S_1,S_2$ are translating sets of a paradoxical decomposition of $G\acts Y$ then they are also translating sets of a paradoxical decomposition of $G\acts X$.
\end{enumerate}
\end{Lemma}

\begin{proof}
(1) Every paradoxical decomposition with translating elements from $H$ is in particular a paradoxical decomposition with translating elements from $G$.

(2) Let $P_1,\ldots, P_m,Q_1,\ldots, Q_n\sub Y$  be a paradoxical decomposition of $G\acts Y$ with translating sets $S_1=\{g_1,\ldots, g_m\}$ and $S_2=\{h_1,\ldots, h_n\}$. Then the inverse images 
 $f^{-1}(P_1),\dots f^{-1}(P_m),f^{-1}(Q_1),\dots,f^{-1}(Q_n)$ form a paradoxical decomposition of $G\acts X$ with the same translating sets.
\end{proof}

\begin{Corollary}\label{transitive}
Let $G\acts X$ be a transitive action and $x\in X$. Let $\Stab_G(x)=\{g\in G:xg=x\}$ be the stabilizer of $x$. Then $G\acts X$ has a paradoxical decomposition with translating sets $S_1,S_2$ if and only if so does the action $G\acts G/\Stab_G(x)$, where $G/\Stab_G(x)$ is the set of right cosets.
\end{Corollary}

\begin{proof}
Let $H=\Stab_G(x)$. 
For every $y\in X$ there exists $g\in G$ such that $y=xg=xHg$. Sending $y$ to $Hg$ results in a $G$-equivariant isomorphism between $X$ and the quotient set $G/H$.
 Thus Lemma \ref{equi}(2) yields the result. 
\end{proof}

\begin{Remark}\label{rem}
Let $H\triangleleft G$ be a normal subgroup. Then if $G\acts G/H$ is paradoxical so is the group $G/H$.
\end{Remark}

\begin{proof}
Every translating element from $G$ can be replaced by its image in $G/H$.
\end{proof}

\begin{Lemma}\label{orbits}
Let $G\acts X$ be a group action.
\begin{enumerate}
\item Let $\{X_{\alpha}\}_{\alpha\in I}$ be a partition of $X$ in which every set is closed under the action of $G$. Then $G\acts X$ has a paradoxical decomposition with translating sets $S_1,S_2$ if and only if for every $\alpha$, the action $G\acts X_{\alpha}$ has a a paradoxical decomposition  with translating sets $S_1,S_2$.
\item $G\acts X$ has a paradoxical decomposition with translating sets $S_1,S_2$ if and only if the same is true for every orbit of the action.
\end{enumerate}
\end{Lemma}

\begin{proof}
$(2)$ follows from $(1)$ by taking the partition of $X$ to be the set of orbits of the action $G\acts X$.

$(1)$ In the one direction, for each $\alpha$, the intersection of the translated sets in a paradoxical decomposition of $X$ with $X_{\alpha}$ forms a paradoxical decomposition of $X_{\alpha}$ with the same translating sets. In the other direction, assume that every $X_{\alpha}$ has a paradoxical decomposition with translating sets 
$S_1=\{g_1,\dots g_m\}$, $S_2=\{h_1,\dots h_n\}$ and translated sets $P_1^{\alpha},\dots P_m^{\alpha}, Q_1^{\alpha}\ldots Q_n^{\alpha}$. Then, the unions $P_i=\bigcup_{\alpha\in I}{P_i^{\alpha}}$ for  $i=1,\dots,m$ and $Q_j=\bigcup_{\alpha\in I}{Q_j^{\alpha}}$ for $j=1,\dots,n$ form a paradoxical decomposition of $X$ with translating sets $S_1$ and $S_2$.
\end{proof}

Combining Lemma \ref{orbits}(2) and Corollary \ref{transitive} we get the following.

\begin{Corollary}\label{Cor}
Let $G\acts X$ be a group action. It has a paradoxical decomposition with translating sets $S_1,S_2$ if and only if for every $x\in X$, the action $G\acts G/\Stab_G(x)$ has a a paradoxical decomposition with these sets as translating sets.
\end{Corollary}

The following are the analogues for group actions of results of \cite{EGS}, proved originally for groups. Remark \ref{1} is the equivalent of \cite[Remark 2.2]{EGS}. Theorem \ref{first} follows from \cite[Lemma 2.5]{EGS} and \cite[Theorem 2.6]{EGS}. Theorem \ref{second} is a reformulation of \cite[Lemma 5.1]{EGS}.

\begin{Remark}\label{1}
If $G\acts X$ has a paradoxical decomposition with translating
sets $S_1,S_2$, then $G\acts X$ also has a paradoxical decomposition with translating
sets $S_1g_1,S_2g_2$ for any given $g_1, g_2\in G$. In particular, we can always assume that $1\in S_1,S_2$.
\end{Remark}

\begin{Theorem}\label{first}
Let $G\acts X$ be a group action. Let $S_1,S_2$ be finite subsets of $G$. Then, the following assertions are equivalent.
\begin{enumerate}
\item $G\acts X$ has a paradoxical decomposition with translating sets $S_1,S_2$.
\item For any pair of finite subsets $A_1,A_2\sub X$, $|A_1S_1^{-1}\cup A_2S_2^{-1}|\ge |A_1|+|A_2|$.
\end{enumerate}
\end{Theorem}

\begin{Theorem}\label{second}
Let $G\acts X$ be a group action and $S=\{a,b,c\}\sub G$. Assume that for any finite $A\sub X$ we have $|AS^{-1}\cup A|\ge 2|A|$. Then $\mathcal T(G\acts X)\le 6$. 
\end{Theorem}

\section{Schreier graphs and automata}\label{automata}

The definitions in this section follow \cite{BOl,MSW,Ol}.

Given a free group $F=\la x_1,x_2,\dots x_m\ra$ and a subgroup $H\le F$, let $\mathcal{G}$ denote the Cayley graph of the action $F \acts F/H$ 
  with respect to the symmetric set $S=\{x_1^{\pm 1},\dots,x_m^{\pm 1}\}$. We will refer to this graph as the Schreier graph of the subgroup $H$. By definition every vertex in the graph has exactly $2m$ outgoing edges, each labeled by a different element of $S$. For every (directed) edge $e$, $e_-$ and $e_+$ will denote the initial and final vertex of $e$ respectively. Note that every edge $e$ has an inverse edge $f$ such that $e_-=f_+,e_+=f_-$ and the labels of $e$ and $f$ are inverses of each other. Sometimes we will refer to $e$ and its inverse as a single geometric edge labeled by a letter $c^{\pm 1}$. A path in $\mathcal {G}$ is a sequence of directed edges $e_1,\dots,e_n$ where for $i<n$ the final vertex of $e_i$ is the initial vertex of $e_{i+1}$. It is said to be reduced if $e_{i+1}\neq e_i^{-1}$ for all $i<n$. A cycle $e_1,\dots,e_n$ is called reduced if it is reduced as a path. That is, $e_n$ might be equal to $e_1^{-1}$ in a reduced cycle. 

Let $o$ be the vertex corresponding to the group $H$ and $\mathcal{C}$ the minimal subgraph of $\mathcal{G}$ containing $o$ and all reduced cycles from it to itself. $\mathcal{C}$ will be called the \emph{Stallings core} of $H$ or simply the \emph{core} of $H$. Sometimes we will refer to the core as the automaton of $H$. Note that the elements of $H$ are exactly those words which in reduced form can be read on a cycle in $\mathcal{C}$ from $o$ to itself. Also, if for some reduced word $w\in F$, the coset $Hw$ belongs to the core of $H$, then there exists $w'\in F$ such that $ww'$ is reduced and $ww'\in H$. Given the core $\mathcal{C}$ of $H$, it is possible to construct from it the Schreier graph of $H$ by attaching appropriate trees at each vertex of $\mathcal{C}$ with less than $2m$ outgoing edges. If such a vertex exists, the group $H$ does not contain any normal subgroup. For this fact and further details see \cite{BOl}.

Given a finite number of elements $p_1,p_2,\dots,p_n\in F$ there is a simple algorithm for the construction of the automaton $\mathcal A$ corresponding to $H=\la p_1,p_2,\dots,p_n\ra$. The first step consists of attaching $n$ cycles to the origin $o$ and labeling them by the words $p_i$. The second step, consists of identifying every two outgoing edges of the same vertex which have the same label, until there are no vertices with two outgoing edges labeled by the same letter.
 At last, vertices of degree one other than the origin are deleted. For further details, see \cite{MSW}. Once $\mathcal A$ is given, it is possible to erase a finite number of edges and get a spanning tree $T$. If $k$ edges were erased, then $H$ is free of rank $k$. In particular, $k\le n$. 
 For this fact and further details, see \cite{Ol}.

\begin{Lemma}\label{deg 4}
Let $F=\la x_1,\dots,x_m\ra$ be a free group of rank $m$ and $p_1,\dots,p_n\in F$. 
\begin{enumerate}
\item Let $\mathcal A$ be the automaton corresponding to the subgroup $H=\la p_1,\dots,p_n\ra$. Then, the origin $o$ has at most $2n$ incoming edges.
\item Let $K\le H$ be a (not necessarily finitely generated) subgroup and $\mathcal A'$ the automaton of $K$. Then, the origin $o'$ of $\mathcal{A}'$ has at most $2n$ incoming edges.
\item Let $M\le F$ be any finitely generated subgroup of infinite index and $\mathcal B$ the automaton corresponding to it. Then, there exists a vertex $v$ in $\mathcal B$ such that $v$ has less than $2m$ incoming edges. 
\end{enumerate}
\end{Lemma}

\begin{proof}
(1) Let $N=\{q_1,\dots,q_k\}$ be a Nielsen reduced set, Nielsen equivalent to $\{p_1,\dots,p_n\}$. In particular $k\le n$ and $N$ freely generates $H$. Thus, every element $w\in H$ has a unique presentation as a word in the elements of $N$ and their inverses. Also, if $q_i^{\epsilon}$ for $\epsilon=\pm 1$ is the last element in the presentation of $w\in H$ then, as a word in the generators of $F$, the last letters of $w$ and $q_i^{\epsilon}$ coincide. Thus, there are at most $2k\le 2n$ possibilities for the last letter of a reduced word in $H$. In particular, the origin of $\mathcal {A}$ has at most $2n$ distinct incoming edges.

(2) If $c$ labels an incoming edge of $o'$ in $\mathcal{A}'$ then $c^{-1}$ labels an outgoing edge and there is a reduced word $w=c^{-1}w'$ in $K$ beginning with $c^{-1}$. Since $K\le H$, the word $w\in H$. Thus $c^{-1}$ labels an outgoing edge of $o$ in $\mathcal {A}$ and $c$ labels an incoming one. Hence the result follows from part $(1)$.

(3) If every vertex in $\mathcal B$ is of degree $2m$ then $\mathcal B$ is the Schreier graph of the action $F\acts F/M$. Since $M$ is finitely generated, the set of vertices of $\mathcal B$ is finite. Thus, $M$ has finite index in $F$, a contradiction.
\end{proof}

\begin{Proposition}\label{prop}
Let $G_n=\la x,y_1,\dots,y_n,z\ra$ be an $n+2$ generated free group. Then for every $p_1,\dots,p_n\in G_n$, if $H=\gamma_2\la p_1,\dots,p_n\ra$ is the derived subgroup of the group they generate, there exists $j\in\{1,\dots,n\}$ such that for all $g\in G_n$ we have $H\cap \la x,y_j\ra ^g=\{1\}$. 
\end{Proposition}

\begin{proof}
By induction on $n$. For $n=1$ for every $p_1\in G_1$ the group $H=\{1\}$ and the proposition holds. Assume the proposition holds for $n$ but not for $n+1$. Let $p_1,\dots,p_{n+1}\in G_{n+1}$ be elements for which the proposition fails. In particular, for $j=n+1$ there exists $g\in G_{n+1}$ and a non trivial word $u(x_1,\dots,x_{n+1})\in \gamma_2\la x_1,\dots,x_{n+1}\ra$, where $\la x_1,\dots,x_{n+1}\ra$ is a free group of rank $n+1$, such that substituting $p_i$ for $x_i$ gives a non trivial element $u=u(p_1,\dots,p_{n+1})\in\gamma_2\la p_1,\dots,p_{n+1}\ra\cap \la x,y_{n+1} \ra ^{g}$.

Let $\pi\colon G_{n+1}\to G_n$ be the homomorphism taking $y_{n+1}$ to $1$ and any other generator of $G_{n+1}$ to its copy in $G_n$. Then $\pi(u)\in \la x \ra ^{\pi(g)}$. Since $\pi(u)=\pi(u(p_1,\dots,p_{n+1}))=u(\pi(p_1),\dots,\pi(p_{n+1}))\in \gamma_2G_n$ we have $\pi(u)=1$. Indeed the intersection $\gamma_2G_n\cap \la x\ra ^{\pi(g)}$ is trivial.
Since $u(x_1,\dots,x_{n+1})$ is a non trivial word, $\pi(p_1),\dots,\pi(p_{n+1})$  are not free generators of the  group $K$ they generate. In particular, if $K$ is free of rank $r$ then $r\le n$. Let $\{q_1,\dots,q_n\}\sub G_n$ be an $n$ element subset which generates $K$. The following claim yields the required contradiction.

\begin{Claim}
The conclusion of Proposition \ref{prop} does not hold for $G_n$ with  the elements $q_1,\dots,q_n$. 
\end{Claim}    

\begin{proof}
Otherwise, for some $j\in\{1,\dots,n\}$ and every $a\in G_n$ we have $\gamma_2\la q_1,\dots,q_n\ra\cap\la x,y_j\ra^{a}=\{1\}$. 
By assumption, there exists $b\in G_{n+1}$ and a non trivial element $v\in \gamma_2\la p_1,\dots,p_{n+1}\ra$ $ \cap \la x,y_j\ra ^{b}$. In particular, $v={b}^{-1}wb$ for a non trivial word $w\in \la x,y_j\ra$. Let $v'=\pi(v)$, then $v'=\pi(b)^{-1}w\pi(b)$ where we now consider the word $w$ as an element of $G_n$. Note that as a word in the letters $x^{\pm 1},y_j^{\pm 1}$, the reduced form of $w$ is not affected by the homomorphism $\pi$. Therefore, since $w$ is not trivial, $v'\neq 1$. On the other hand, $v'\in \gamma_2\la \pi(p_1),\dots,\pi(p_{n+1})\ra=\gamma_2\la q_1,\dots,q_n\ra$. Therefore $\gamma_2\la q_1,\dots,q_n\ra\cap \la x,y_j\ra ^{\pi(b)}$ is not trivial. A contradiction. 
\end{proof}    
\end{proof}

\begin{Corollary}\label{cyc2}
Let $G_n=\la x,y_1,\dots,y_n,z\ra$ be a free group of rank $n+2$ and $p_1,\dots,p_n\in G_n$. Let $\mathcal{A}$ be the automaton corresponding to the subgroup $H=\gamma_2\la p_1,\dots,p_n\ra$. Then there exists $j\in\{1,\dots,n\}$ such that there are no reduced non trivial cycles in $\mathcal {A}$ labeled by elements of $\la x,y_j\ra$. 
\end{Corollary}

\begin{proof}
Let $j\in\{1,\dots,n\}$ be an index for which the conclusion of Proposition \ref{prop} is satisfied. 
Assume by contradiction that $s$ is a reduced non trivial cycle in $\mathcal {A}$ labeled by a word in $\la x,y_j\ra$ and let $v$ be the initial (and final) vertex of $s$. There exists $g\in G$ such that $v$ represents the coset $Hg$. Thus, if $w$ is the label of $s$, $Hgw=Hg$ implies that $w\in H^g\cap \la x, y_j\ra$. Then $w^{g^{-1}}\in H\cap \la x,y_j\ra^{g^{-1}}$ is a non trivial element, a contradiction.
\end{proof}

\section{Construction of group actions with a given Tarski number}\label{construction}

In this section we prove Theorem \ref{thm:1}. Let $F=G_n=\la x,y_1,\dots,y_n,z\ra$ be an $n+2$ generated free group for $n\in \mathbb{N}$. We will construct a subgroup $H$ for which the action $F\acts F/H$ is faithful and has Tarski number $n+3$.  $H$ will be defined by means of its core.

Let $(p_{i,1},\dots,p_{i,n})_{ i\in\mathbb{N}}$ be an enumeration of all the $n$-tuples of elements of $G_n$. For each $i$ let $\mathcal A_i$ be the automaton representing the subgroup $K_i=\gamma_2\la p_{i,1},\dots, p_{i,n}\ra$ and let $o_i$ be its origin. By Lemma \ref{deg 4}, $o_i$ has at most $2n$ incoming edges. Thus, there exists a letter $c_i\notin\{z,z^{-1}\}$ different than the labels of all incoming edges of $o_i$.

We define the core $\mathcal C$ of $H$ in the following way (for an illustration, see Figure \ref{fig:graph}).
Let $o$ be the origin of $\mathcal C$ and $e_1,e_2,\dots$ an infinite sequence of edges, all labeled by $z$, such that ${e_1}_-=o$ and for all $i$ we have ${e_i}_+={e_{i+1}}_-$. Since the letters $c_i\notin\{z,z^{-1}\}$, for each $i$ it is possible to attach to ${e_i}_+$ an outgoing edge labeled by $c_i$. To its head vertex one can attach the automaton $\mathcal A_i$ by identifying $o_i$ with the vertex in question. Indeed, the choice of letters $c_i$ guarantees that no cancellation occurs in $\mathcal C$.

Clearly, if $H$ is the group represented by $\mathcal C$, then $H=\bigcup_{i\in\mathbb{N}}\gamma_2\la p_{i,1},\dots,p_{i,n}\ra^{(z^ic_i)^{-1}}$.
By construction, the origin $o$ has degree $1$ in $\mathcal C$. 
 In particular, there are vertices in $\mathcal{C}$ of degree smaller than $2(n+2)$ and $H$ does not contain any normal subgroup \cite{BOl}.

\begin{figure}[b]
\includegraphics[width=1\columnwidth]{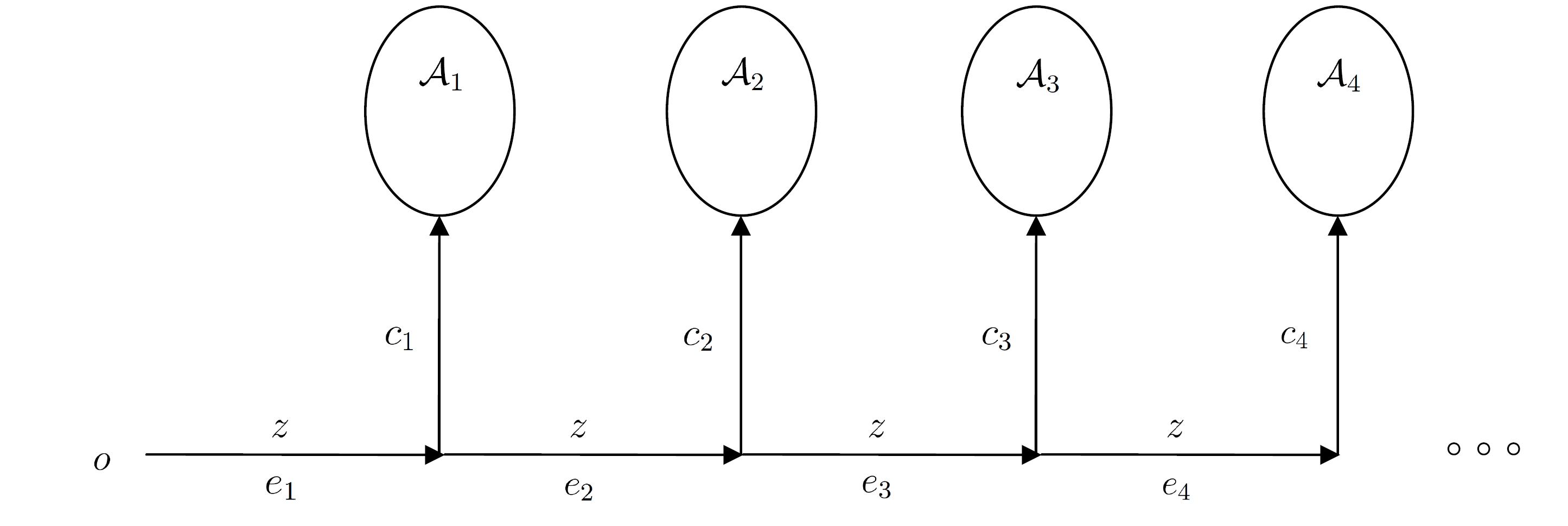}
\caption{The core of $H$}
    \label{fig:graph}
\end{figure}

Let $\mathcal{G}$ be the Schreier graph of the action $F\acts F/H$. The graph $\mathcal{G}$ can be obtained from $\mathcal{C}$ by attaching trees to every vertex of $\mathcal{C}$ of degree less than $2(n+2)$. The action $F\acts F/H$ can be described in terms of the action of $F$ on the graph $\mathcal{G}$.

\begin{Lemma}\label{u}
Let $v$ be a vertex in $\mathcal{G}$. 
There is at most one automaton $\mathcal A_m$ to which one can get from $v$ via a path whose label does not include the letter $z^{\pm 1}$. 
\end{Lemma}

\begin{proof}
Clearly, there is no path between two different automata $\mathcal A_l$ and $\mathcal A_r$ which does not cross an edge labeled by $z^{\pm 1}$. Assume that $t_1,t_2$ are paths from $v$ to two distinct automata $\mathcal A_{\alpha}$ and $\mathcal A_{\beta}$, such that both $t_1$ and $t_2$ do not cross any edge labeled by $z^{\pm 1}$. Then the path $t_1^{-1}t_2$ connects $\mathcal A_{\alpha}$ and $\mathcal A_{\beta}$ and does not contain the letter $z^{\pm 1}$.
\end{proof}

Let $\{X_j\}_{j=1}^n$ be a partition of the set of automata $\{\mathcal A_i\}_{i\in\mathbb{N}}$, where $\mathcal A_i\in X_j$ if and only if $j$ is the smallest index which satisfies the conclusion of Corollary \ref{cyc2} for the automaton $\mathcal A_i$.  
By Lemma \ref{u}, for each vertex $v$ of $\mathcal{G}$ there exists at most one automaton to which it is possible to get via a path not including the letter $z^{\pm 1}$. Thus, it is possible define a partition of the vertex set of $\mathcal{G}$ to $n$ sets $\{Y_j\}_{j=1}^{n}$ in the following way. For a vertex $v$, if $\mathcal A_m$ is an automaton reachable from $v$ via a path not containing the letter $z^{\pm 1}$ and  $\mathcal A_m$ belongs to $X_j$ for some $j$, then $v$ will belong to $Y_j$ for the same $j$. If no automaton is reachable from $v$ via such a path, $v$ will belong in $Y_{1}$. 
 Note, that each of the sets in the partition is closed under the action of $\la x,y_1,\dots,y_n\ra$.

\begin{Lemma}\label{cyc}
For $j=1,\dots,n$ the group $\la x,y_j\ra$ acts freely on $Y_j$.   
\end{Lemma}

\begin{proof}
Let $j\in\{1,\dots,n\}$ and $v$ be a vertex of $Y_j$. Assume by contradiction that $w\in \la x,y_j\ra $ is a reduced non trivial word stabilizing $v$. Then $w$ labels a reduced non trivial cycle $s$ from $v$ to itself in $\mathcal{G}$. Since $s$ is non trivial, it must contain as a subpath a reduced non trivial cycle $s'$ through some automaton $\mathcal A_m$. Note that $\mathcal A_m$ is reachable from $v$ via a subpath of $s$, which by definition does not contain the letter $z^{\pm 1}$. 
 Therefore, $v\in Y_j$ implies that $\mathcal A_m$ contains no reduced non trivial cycle labeled by a word in $\la x,y_j\ra$, a contradiction.
\end{proof}

\begin{Lemma}\label{>=}
The Tarski number of the action of $F$ on $\mathcal{G}$  is at least $n+3$. 
\end{Lemma}

\begin{proof}
Assume by contradiction that the action has Tarski number at most $n+2$ and let $S_1, S_2$ be translating sets of a paradoxical decomposition with $|S_1|+|S_2|\le n+2$. By Remark \ref{1}, we can assume that $1\in S_1\cap S_2$. Then, $S=(S_1\cup S_2)\setminus\{1\}$ is a set of $n$ elements at most. Let $K$ be the subgroup it generates.   Then $K\acts \mathcal{G}$ has a paradoxical decomposition with translating sets $S_1,S_2$.  Let $p_1,\dots,p_n$ be the elements of $S$ (possibly with repetitions) and assume the $n$-tuple $(p_1,\dots,p_n)$ was enumerated as tuple number $m$. Let $o_m$ be the origin of the automaton $\mathcal A_m$. By Corollary \ref{Cor}, $K\acts K/\Stab_K(o_m)$ has a paradoxical decomposition with translating sets $S_1,S_2$. Since $\gamma_2(K)\sub \Stab_{K}(o_m)$, Lemma \ref{equi}(2) implies that the  same is true for the action $K\acts K/\gamma_2(K)$. In particular, this action is paradoxical. By Remark \ref{rem}, the group $K/\gamma_2(K)$ is paradoxical, in contradiction to it being abelian.
\end{proof}

\begin{Lemma}\label{<=}
Let $F'=\la x,y_1,\dots,y_n\ra$. Then $F' \acts \mathcal{G}$ has a paradoxical decomposition with translating sets $S_1=\{1,x\}$, $S_2=\{1,y_1,\dots,y_n\}$. In particular, $\mathcal T(F'\acts \mathcal{G})\le n+3$. 
\end{Lemma}

\begin{proof}
$\mathcal{G}$ is the disjoint union of the sets $Y_j$ for $j=1,\dots,n$ where each of the sets is closed under the action of $F'$. 
 By Lemma \ref{cyc}, for each $j$, the action of $\la x,y_{j}\ra $ on $Y_j$
is free. Thus by Corollary \ref{x,y}, $Y_j$ has a paradoxical decomposition with translating sets $\{1,x\}$ and $\{1,y_{j}\}$. By adding empty sets to the decomposition, we get that every $Y_j$ has a paradoxical decomposition with translating sets $S_1$ and $S_2$. Thus Lemma \ref{orbits}(1) yields the result.
\end{proof}

By Lemma \ref{equi}(1), $\mathcal T(F\acts \mathcal{G})\le \mathcal T(F'\acts \mathcal{G})\le n+3$. Thus, by Lemma \ref{>=}, $$\mathcal T(F\acts F/H)=\mathcal T(F\acts \mathcal{G})=n+3.$$ $\hfill\qed$

\begin{Remark}
For every $k,l\in\mathbb{N}$ such that $k+l=n+1$ it is possible to rename the first $n+1$ generators $x,y_1,\dots y_n$ of $F=G_n$ by $x_1,\dots,x_k$, $y_1,\dots,y_l$. Then, for the subgroup $H$ constructed above, $F\acts F/H$ has a paradoxical decomposition with translating sets $S_1=\{1,x_1,\dots,x_k\}$ and $S_2=\{1,y_1,\dots,y_l\}$. Indeed, the only necessary change is to Proposition \ref{prop}.
\end{Remark} 

\begin{Proposition}\label{prop2}
Let $k,l\in\mathbb{N}$ and $G_{k,l}=\la x_1,\dots,x_k,y_1,\dots,y_l,z\ra$ be a $k+l+1$ generated free group. Then for every $p_1,\dots,p_{k+l-1}\in G_{k,l}$, if $H=\gamma_2\la p_1,\dots,p_{k+l-1}\ra$ is the derived subgroup of the group they generate, there exist $i\in\{1,\dots,k\}$ and $j\in \{1,\dots,l\}$ such that for all $g\in G_{k,l}$  we have $H\cap \la x_i,y_j\ra ^g=\{1\}$. 
\end{Proposition}

\begin{proof}
By induction on $k$. The case $k=1$ is Proposition \ref{prop}. Assume the proposition holds for $k$ (and every $l$) but not for $k+1$. Then there exists $l\in\mathbb{N}$ such that the proposition fails for $G_{k+1,l}$. The
reduction to the case $G_{k,l}$ follows the same argument as that in Proposition \ref{prop}. Here the homomorphism $\pi\colon G_{k+1,l}\to G_{k,l}$ maps $x_{k+1}$ to the identity and any other generator to its copy.
\end{proof}

\begin{Corollary}
Let $k\ge 4$. There exists a finitely generated free group $F$ and a faithful transitive group action $F\acts X$, such that $\mathcal T(F\acts X)=k$ and for all $m,n\ge 2$ such that $m+n=k$ the action $F\acts X$ has a paradoxical decomposition with translating sets $S_1,S_2$ such that $|S_1|=m$ and $|S_2|=n$.
\end{Corollary}

Note that nothing similar is known for groups. Indeed, we don't have an example of a group with Tarski number $k$ which has two paradoxical decompositions, one with translating sets of size $m_1$ and $n_1$ and the other with translating sets of size $m_2$ and $n_2$, such that for $i=1,2$ we have $m_i+n_i=k$ and $\{m_1,n_1\}\neq \{m_2,n_2\}$.

\section{Unbounded Tarski numbers}\label{qi}

In what follows, $p$ will be a fixed prime number. Let $F$ be a finitely generated non abelian free group.
Let $\{\omega_n F\}_{n\in\dbN}$ be the {\it Zassenhaus $p$-filtration} of $F$
defined by $\omega_n F=\prod_{i\cdot p^j\geq n}(\gamma_i F)^{p^j}$.
It is easy to see that $\{\omega_n F\}$ is a descending chain of
normal subgroups of $p$-power index in $F$.
Moreover, $\{\omega_n F\}$ is a base for the pro-$p$ topology on $F$, so in particular, $F$ being residually-$p$ implies that $\cap \omega_n F=\{1\}$. 
 It follows that for any $n\in\mathbb{N}$ there exists $m(n)\in\mathbb{N}$ such that the reduced form of any element of $\omega_{m(n)}F$ is of length $\ge 12n$. Clearly, the index $[F:\omega_{m(n)}F]>n$.
 Thus, by the Schreier index formula, $\omega_{m(n)}F$ is free of rank $>n$. In particular, every $n$ elements $p_1,\dots,p_n\in\omega_{m(n)}F$ generate a subgroup of infinite index inside $\omega_{m(n)}F$ and thus inside $F$.

Theorem \ref{thm:3} is a straightforward corollary of the following theorem.

\begin{Theorem} \label{thm:32}
Let $F=\la x,y,z\ra$ and for each $n\in\mathbb{N}$ let $m(n)$ be as described above. There exists $H\le F$ with the following properties.
\begin{enumerate}
\item $H$ does not contain a non trivial normal subgroup of $F$.
\item For each $n\in \mathbb{N}$, $\mathcal T(\omega_{m(n)}F\acts F/H)\ge n+3$. 
\item $\mathcal T(F\acts F/H)=6$.
\end{enumerate}
\end{Theorem}

\begin{proof}
For each $n\in\mathbb{N}$, let $(p_{i,1},\dots,p_{i,n})_{i\in\mathbb{N}}$ be an enumeration of the elements of $\omega_{m(n)}F$. 
For each $n,i\in\mathbb{N}$ let $\mathcal A_{(n,i)}$ be the automaton corresponding to the subgroup generated by the elements of the $n$-tuple  $(p_{i,1},\dots,p_{i,n})$. 
By Lemma \ref{deg 4}(3) there exists a vertex $o_{(n,i)}'$ in $\mathcal A_{(n,i)}$ with less than $6$ incoming edges. Let $c_{(n,i)}$ be a letter distinct from the labels of all the incoming edges of $o_{(n,i)}'$. 
Let $\alpha(k)$ for $k=1,2,\dots$ be an enumeration of all the pairs $(n,i)\in\mathbb{N}\times\mathbb{N}$.

The construction of the core $\mathcal{C}$ of $H$ will be similar to the construction used in section \ref{construction}. Let $o$ be the origin of $\mathcal C$ and $e_1,e_2,\dots$ be an infinite sequence of edges such that ${e_1}_-=o$ and for all $k$ we have ${e_k}_+={e_{k+1}}_-$. It is possible to label the edges $e_k$ inductively such that if $l(e_k)$ is the label of $e_k$,  then $l(e_1)\neq c_{\alpha(1)}^{-1}$ and for each $k>1$, the label $l(e_k)\notin \{l(e_{k-1})^{-1},c_{\alpha(k-1)},c_{\alpha(k)}^{-1}\}$. The choice of the labels of $e_k$ means that for all $k$, one can attach to ${e_k}_+$ an outgoing edge labeled by $c_{\alpha(k)}$. To its head vertex, it is possible to attach the automaton $\mathcal A_{\alpha(k)}$ by identifying $o'_{\alpha(k)}$ with the vertex in question. Indeed, the choice of letters $c_{\alpha(k)}$ guarantees than no cancellation occurs in $\mathcal C$. 
Denote by $\mathcal G$ the Schreier graph of the group $H$ represented by $\mathcal C$.

\begin{Lemma}
For each $n\in\mathbb{N}$ we have $\mathcal T(\omega_{m(n)}F\acts \mathcal G)\ge n+3$.
\end{Lemma}

\begin{proof}
Similar to the proof of Lemma \ref{>=}. If $K$ is an $n$-generated subgroup of $\omega_{m(n)}F$, it fixes a point of $\mathcal{G}$. In particular, the action $K\acts \mathcal G$ is not paradoxical.
\end{proof}

\begin{Lemma}\label{tree}
Let $n\in\mathbb{N}$. Let $p_1,\dots,p_n\in \omega_{m(n)}F$ and $\mathcal A$ be the automaton corresponding to the group $K$ they generate as a subgroup of $F$. 
\begin{enumerate}
\item There exists a spanning tree $T$ in $\mathcal A$ such that every vertex in $\mathcal A$ loses at most one of the edges adjacent to it in the transition from $\mathcal A$ to $T$.
\item $\mathcal A$ does not contain loops.
\end{enumerate}
\end{Lemma}

\begin{proof}
(1) As mentioned in the introduction, in order to construct a spanning tree of $\mathcal A$ we have to erase at most $n$ edges from $\mathcal A$.
Assume $i$ edges, $i \in\{0,\dots,n-1\}$, were already erased and no two of them are adjacent to the same vertex.
If the resulting graph is a tree, we are done. Otherwise, let $e$ be an edge whose removal would not affect the connectivity of the graph. Let $v$ be its initial vertex and $s$ a reduced cycle from $v$ to itself which starts with the edge $e$ and does not visit any vertex other than $v$ twice.
 Then, the removal of any edge of $s$ would not affect the connectivity of $\mathcal A$. If $v$ corresponds to the coset $Kg$ and $w$ is the label of the cycle $s$, then $w\in K^g\sub \omega_{m(n)}F$. As such, the length of $w$, and of the cycle $s$, is at least $12n$. Until now, at most $n-1$ edges have been erased. Each of them is adjacent to at most $2$ vertices. Each of the $2(n-1)$ vertices in question is adjacent to at most $6$ edges. Thus there are at most $12(n-1)$ edges adjacent to vertices which have already lost an edge. As such, at least one edge on the cycle $s$ is not one of these edges and one can erase it to complete the induction.

(2) As demonstrated in the proof of part $(1)$, all reduced non trivial cycles of $\mathcal {A}$ are of length $\ge 12n$.
\end{proof}

\begin{Lemma}\label{<=6}
Let $S=\{x,y,z\}$. Then for any finite set $A$ of vertices of $\mathcal {G}$, we have $|AS^{-1}\cup A|\ge 2|A|$. In particular, by Theorem \ref{second}, $\mathcal T(F\acts \mathcal {G})\le 6$.
\end{Lemma}

\begin{proof}
From each of the automata $\mathcal A_{(n,i)}$ attached during the construction of the core $\mathcal C$, it is possible to erase at most $n$ edges such that the resulting spanning tree of the automata satisfies the conclusion of Lemma \ref{tree}(1). Let $\mathcal T$ be the graph obtained in this way from the graph $\mathcal G$. Clearly, $\mathcal T$ is a tree.
Lemma \ref{tree}(2) implies that there are no loops in $\mathcal{G}$. Thus, every vertex in $\mathcal G$ is adjacent to $6$ distinct unoriented edges. The choice of the tree $\mathcal T$ implies that each vertex in $\mathcal T$ is adjacent to at least $5$ edges. Thus, considering orientation, every vertex of $\mathcal T$ has at least two incoming edges labeled by elements of $S$.

Let $A$ be a finite set of vertices of $\mathcal{G}$. Let $E$ be the set of all oriented edges  $e=(as^{-1},a)$ such that $a\in A$, $s\in S$ and the unoriented edge $\{as^{-1},a\}$ lies in $\mathcal T$. From the above, $E$ contains at least $2|A|$ edges and no pair of opposite ones. The endpoints of edges in $E$ lie in the set $A \cup AS^{-1}$.
Let $\Lambda$ be the unoriented graph with vertex set $A \cup AS^{-1}$
and edge set $E$ (with forgotten orientation). Then $\Lambda$ is a subgraph of $\mathcal T$;
in particular $\Lambda$ is a (finite) forest. Hence, if $V(\Lambda)$ and $E(\Lambda)$ denote the sets of vertices and edges of $\Lambda$, respectively, then
$$|A \cup AS^{-1}|=|V(\Lambda)|>|E(\Lambda)|=|E|\geq 2|A|,$$
 as desired.
\end{proof}

\begin{Lemma}\label{=6}
$\mathcal T(F\acts \mathcal{G})= 6$.
\end{Lemma}

\begin{proof}
By contradiction, let $S_1=\{1,a\},S_2=\{1,b,c\}$ (possibly with $b$=$c$) be translating sets of a paradoxical decomposition of $F\acts \mathcal G$.
For $r=p^{m(3)}$, let $p_1=a^{r},p_2=(a^{b})^{r}$ and $p_3=(a^{c})^{r}$. Then $p_1,p_2,p_3\in\omega_{m(3)}F$. 
Let $\mathcal A$ be the automaton corresponding to the group $K$ generated by $p_1,p_2,p_3$ and $o_{\mathcal{A}}$ its origin. $\mathcal A$ is attached to the core of $H$ by some vertex of $\mathcal A$.
Let $A_1,A_2$ be finite sets of vertices of $\mathcal {G}$ defined as follows. $A_1=o_{\mathcal{A}}\cdot\{a^j,b^{-1}a^j,c^{-1}a^j:0\le j\le r-1\}$ and $A_2=\{o_{\mathcal{A}}\}$.
A simple calculation shows that $A_1S_1^{-1}=A_1\{1,a^{-1}\}=A_1$ (for a visual illustration, see Figure \ref{fig:2}). Clearly, $A_2S_2^{-1}\sub A_1$. Thus,
 $$|A_1S_1^{-1}\cup A_2S_2^{-1}|=|A_1|<|A_1|+|A_2|,$$ 
by contradiction to the implication $(1)\Rightarrow (2)$ of Theorem \ref{first}.

\begin{figure}[h]
\includegraphics[width=0.75\columnwidth]{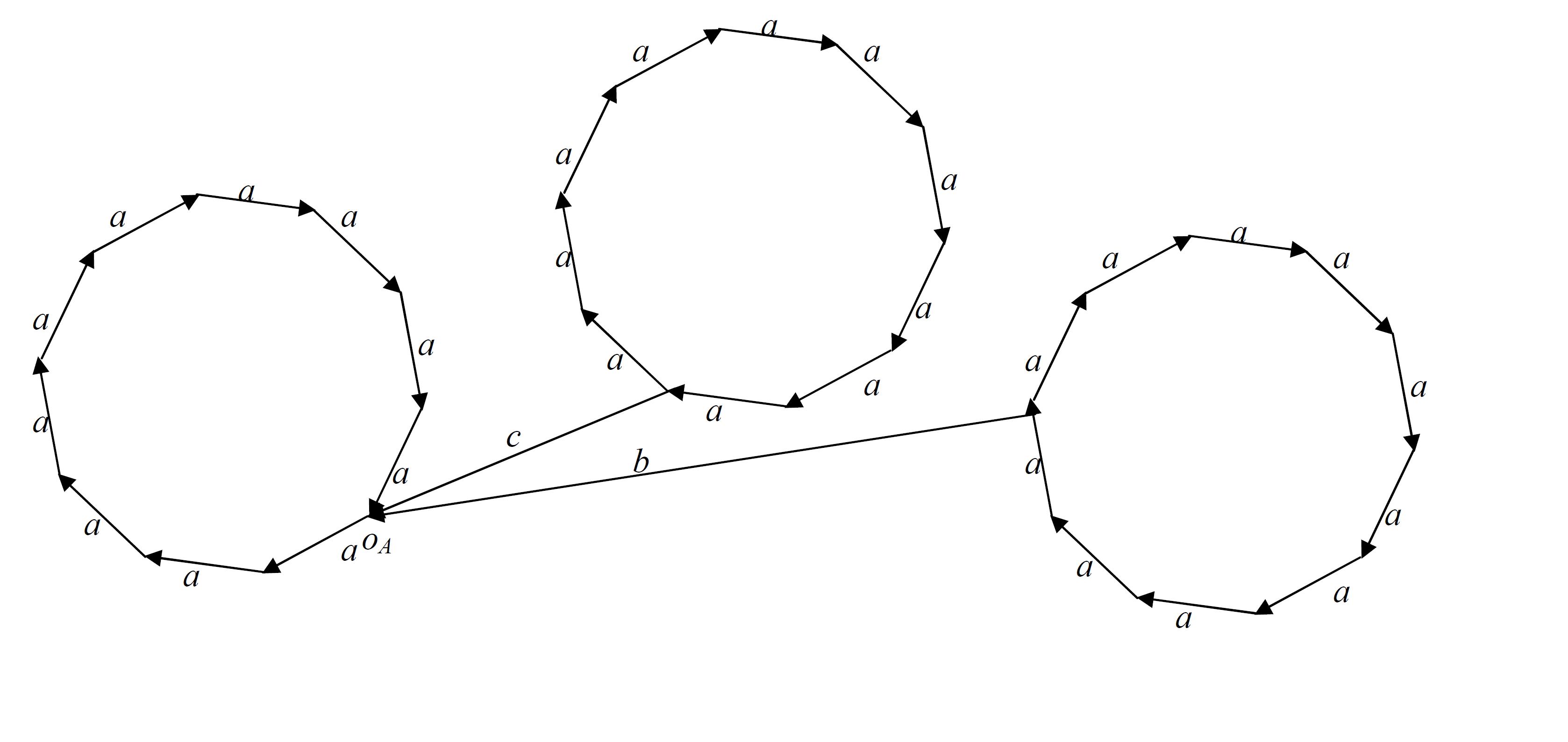}
\caption{The automaton $\mathcal A$. The set $A_1$ is the set of all the vertices in the figure.}
    \label{fig:2}
\end{figure}

\end{proof}
\end{proof}

\end{document}